\documentclass[11pt,reqno,letterpaper]{amsart}

\usepackage{amsmath,amsthm,amsfonts,amssymb,mathtools,bbm,enumerate}

\usepackage{bookmark}
\usepackage{hyperref}
\hypersetup{pdfstartview={FitH}}

\theoremstyle{plain}
\newtheorem{theorem}{Theorem}

\newtheorem{proposition}[theorem]{Proposition}

\theoremstyle{definition}

\theoremstyle{remark}

\numberwithin{equation}{section}

\renewcommand{\leq}{\leqslant}
\renewcommand{\geq}{\geqslant}

\usepackage{graphicx}

\usepackage{xcolor}
\newcommand{\comment}[1]{}

\newcommand{\R}{\mathbb{R}}

\begin{document}

\title[A strong-type Furstenberg--S\'{a}rk\"{o}zy theorem]{A strong-type Furstenberg--S\'{a}rk\"{o}zy theorem for sets of positive measure}

\author[P. Durcik]{Polona Durcik}
\address{Polona Durcik\\
        Schmid College of Science and Technology\\
        Chapman University\\
        One University Drive\\
        Orange, CA 92866, USA}
\email{durcik@chapman.edu}

\author[V. Kova\v{c}]{Vjekoslav Kova\v{c}}
\address{Vjekoslav Kova\v{c}\\
        Department of Mathematics, Faculty of Science\\
        University of Zagreb\\
        Bijeni\v{c}ka cesta 30\\
        10000 Zagreb, Croatia}
\email{vjekovac@math.hr}

\author[M. Stip\v{c}i\'{c}]{Mario Stip\v{c}i\'{c}}
\address{Mario Stip\v{c}i\'{c}\\
        Schmid College of Science and Technology\\
        Chapman University\\
        One University Drive\\
        Orange, CA 92866, USA}
\email{stipcic@chapman.edu}

 %Harmonic analysis on Euclidean spaces - Maximal functions, Littlewood-Paley theory

%\keywords{...}

\begin{abstract}
For every $\beta\in(0,\infty)$, $\beta\neq 1$ we prove that a positive measure subset $A$ of the unit square contains a point $(x_0,y_0)$ such that $A$ nontrivially intersects curves $y-y_0 = a (x-x_0)^\beta$ for a whole interval $I\subseteq(0,\infty)$ of parameters $a\in I$.
A classical Nikodym set counterexample prevents one to take $\beta=1$, which is the case of straight lines.
Moreover, for a planar set $A$ of positive density we show that the interval $I$ can be arbitrarily large on the logarithmic scale.
These results can be thought of as Bourgain-style large-set variants of a recent continuous-parameter S\'{a}rk\"{o}zy-type theorem by Kuca, Orponen, and Sahlsten.
\end{abstract}

\maketitle

%%%%%%%%%%%%%%%%%%%%%%%%%%%%%%%%%%%%%%%%%%%%%%%%

\section{Introduction}

Geometric measure theory often tries to identify patters in sufficiently large, but otherwise arbitrary, measurable sets. Recently, nonlinear or curved patterns have begun to attract much attention \cite{Bou88,HLP16,DGR19,CDR21,CGL21,FGP22,Kov20,KOS21,KMPW22,DR22}; most of these references will be discussed below. In this note we follow one of the many opened lines of research. 

Kuca, Orponen, and Sahlsten \cite{KOS21} showed that there exists $\varepsilon>0$ with the following property: every compact set $K\subseteq\mathbb{R}^2$ with Hausdorff dimension at least $2-\varepsilon$ necessarily contains a pair of points of the form
\begin{equation}\label{eq:pointparab1}
(x,y), \ (x,y) + (u,u^2)
\end{equation}
for some $u\neq0$. We can imagine that we started from a point $(x,y)\in K$, translated the parabola $v=u^2$ so that its vertex falls into $(x,y)$, and moved along that parabola to find another point in the set $K$; see Figure~\ref{fig:oneparab}.
Their result can be thought of as a continuous-parameter analogue of the classical Furstenberg--S\'{a}rk\"{o}zy theorem \cite{Fur77,Sar78}, on $\mathbb{R}^2$ instead of $\mathbb{Z}$.
Parabola cannot be replaced with a vertical straight line (see the comments in \cite{KOS21}); curvature is crucial.

\begin{figure}[ht]
\includegraphics[width=0.6\linewidth]{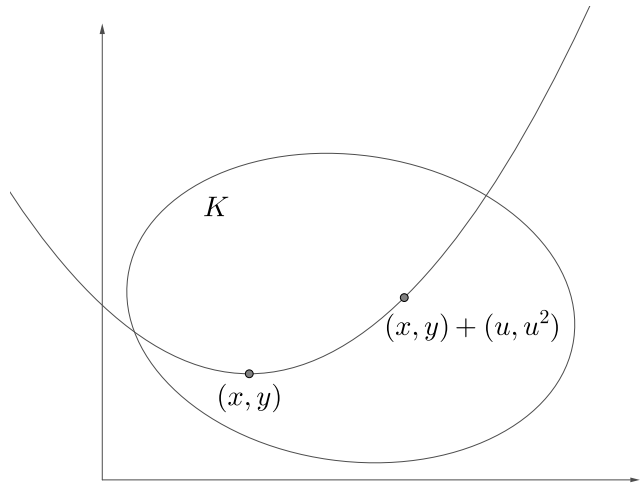}
\caption{The two-point pattern inside the set.}
\label{fig:oneparab}
\end{figure}

The authors of \cite{KOS21} mention that a set $A\subseteq[0,1]^2$ of Lebesgue measure at least $0<\delta\leq1/2$ contains a pair of points \eqref{eq:pointparab1} that also satisfy the \emph{gap bound} 
\[ |u|\geq\exp(-\exp(\delta^{-C})) \] 
for some absolute constant $C$.
This property is seen either by an easy adaptation of Bourgain's argument from \cite{Bou88} for quadratic progressions 
\[ x, \ x+z, \ x+z^2, \]
or by merely considering the last two points of the three-point quadratic \emph{corner} 
\[ (x,y), \ (x+z,y), \ (x,y+z^2), \]
studied by Christ, Roos, and one of the present authors \cite[Theorem~4]{CDR21}.
A gap bound is needed in order to have a nontrivial result, as the Steinhaus theorem would identify sufficiently small copies of any finite configuration inside a set of positive measure.
More on polynomial patterns like these can be found in recent preprints \cite{KMPW22} and \cite{DR22}.

\smallskip
It is natural to wonder if sets $A\subseteq[0,1]^2$ of positive measure also possess some stronger property of the Furstenberg--S\'{a}rk\"{o}zy type.
For instance, we can consider many parabolas $v=au^2$ with their vertex translated to the point $(x,y)$.
Reasoning from the previous paragraph applies equally well for any fixed $a>0$ to the vertically scaled set, giving a well-separated pair of points 
\begin{equation}\label{eq:pointparabmany}
(x,y), \ (x,y)+(u,au^2)
\end{equation}
in the set $A$.
However, it is not   obvious if there exists a common starting point $(x,y)\in A$ from which we could move along ``many'' parabolas and always find points in the set $A$; see Figure~\ref{fig:manyparab}. 
This is the content of our main theorem below and here by \emph{many} we mean a whole ``beam'' of parabolas with parameter $a$ running over a non-degenerate interval $I$.
In fact, a parabola can be replaced with any power curve $v=au^\beta$, for a fixed $\beta\neq 1$ and a varying $a>0$.

\begin{figure}[ht]
\includegraphics[width=0.6\linewidth]{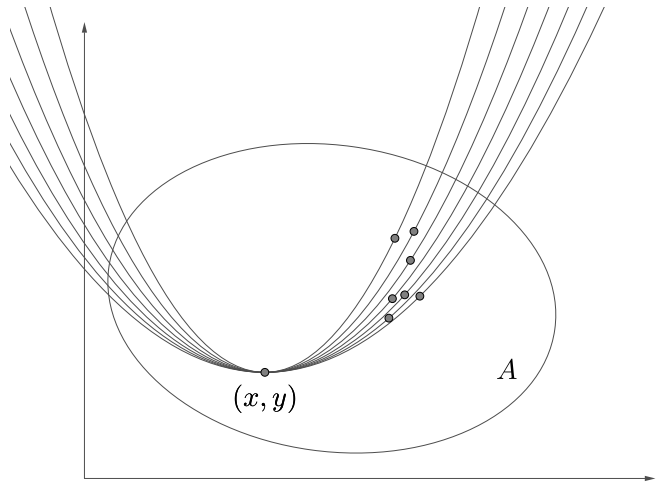}
\caption{Points in the set along many parabolas.}
\label{fig:manyparab}
\end{figure}

Here is the main result of the paper.
Let $|E|$ denote the Lebesgue measure of a measurable set $E\subseteq\mathbb{R}^2$.

\begin{theorem}\label{thm:main1}
For a given $\beta\in(0,\infty)$, $\beta\neq 1$ there exists a finite constant $C\geq1$ with the following property: for every $0<\delta\leq1/2$ and every measurable set $A\subseteq[0,1]^2$ of Lebesgue measure $|A|$ at least $\delta$ there exist a point $(x,y)\in A$ and an interval $I\subseteq(0,\infty)$ such that
\[ \exp(-\delta^{-C})\leq \inf I < \sup I\leq \exp(\delta^{-C}), \]
\[ |I| \geq \exp(-\delta^{-C}), \]
and that for every $a\in I$ the set $A$ intersects the arc of the power curve
\[ \big\{ (x,y) + \big(u, a u^\beta\big) : \exp(-\delta^{-C})\leq u\leq \exp(\delta^{-C}) \big\}. \]
\end{theorem}

The following short argument shows that Theorem~\ref{thm:main1} fails in the limiting case $\beta=1$, i.e., when the power curves are replaced with straight lines through $(x,y)$.
Let $N\subseteq[0,1]^2$ be a \emph{Nikodym set}, which is a set of full Lebesgue measure such that through every point of $N$ one can draw a line that intersects $N$ only at a single  point; let us call such lines \emph{exceptional}.
If $\mathcal{R}_{\alpha}\colon\mathbb{R}^2\to\mathbb{R}^2$ denotes the rotation about  the point $(1/2,1/2)$ by the angle $\alpha$, while $\mathcal{D}_c\colon\mathbb{R}^2\to\mathbb{R}^2$ denotes the dilation centered at $(1/2,1/2)$ by the factor $c>0$, then
\begin{equation}\label{eq:Nikodym_dilates}
A := \bigg( \bigcap_{\alpha\in[0,2\pi)\cap\mathbb{Q}} \mathcal{D}_{\sqrt{2}}\mathcal{R}_{\alpha} N \bigg) \cap [0,1]^2
\end{equation}
is a Nikodym set such that its exceptional lines determine a dense set of directions through each of its points. In particular, there can be no beam of lines 
\[ \big\{(x,y) + (u,au) : u\in\mathbb{R}\big\}, \quad a\in I, \quad I\subseteq(0,\infty) \text{ an interval}, \]
through any point $(x,y)\in A$ that would non-trivially intersect $A$ for each $a\in I$, as required in Theorem~\ref{thm:main1}. 
In fact, Davies \cite{Dav52} has already constructed a Nikodym set whose exceptional lines though each of its points form both dense and uncountable sets of directions.
On the other hand, if we repeat the simple construction \eqref{eq:Nikodym_dilates} starting with a Nikodym-type set found by Chang, Cs\"{o}rnyei, H\'{e}ra, and Keleti \cite[Corollary~1.2]{CCHK18}, then we can also rule out curves composed of countably many pieces of straight lines.

\smallskip
Finally, it is also legitimate to ask if an even stronger result holds for ``really large'' sets, namely for the sets $A\subseteq\mathbb{R}^2$ that occupy a positive ``share'' of the plane. Recall that the \emph{upper Banach density} of a measurable set $A$ is defined as
\[ \overline{\delta}(A) := \limsup_{R\to\infty} \sup_{(x,y)\in\mathbb{R}^2} \frac{\big|A\cap\big([x-R,x+R]\times[y-R,y+R]\big)\big|}{4R^2}. \]

\begin{theorem}\label{thm:main2}
For a given $\beta>1$ (resp.\@ $0<\beta<1$) and a measurable set $A\subseteq\mathbb{R}^2$ with $\overline{\delta}(A)>0$ there is a number $a_0\in(0,\infty)$ with the following property: for every $a_1$ satisfying $0<a_1<a_0$ (resp.\@ $a_1>a_0$) there exist a point $(x,y)\in A$ such that for every $a\in\mathbb{R}$ satisfying $a_1\leq a\leq a_0$ (resp.\@ $a_0\leq a\leq a_1$) the set $A$ intersects the power curve
\[ \big\{ (x,y) + \big(u, a u^\beta\big) : u\in(0,\infty) \big\}. \]
\end{theorem}

In comparison with Theorem~\ref{thm:main1}, an improvement coming from Theorem~\ref{thm:main2} is in the fact that the interval $I=[a_1,a_0]$ (resp.\@ $I=[a_0,a_1]$) can have an arbitrarily small (resp.\@ large) left (resp.\@ right) endpoint $a_1$. It is not clear to us if the latter result also holds with $I=(0,\infty)$; this extension would probably be very difficult to prove. 
Our proof will rely on Bourgain's \emph{dyadic pigeonholing}   in the parameter $a$, and as such it is unable to assert anything for every single value of $a\in(0,\infty)$.
Thus, it is not coincidental that Theorem~\ref{thm:main2} is quite reminiscent of the so-called \emph{pinned distances theorem} of Bourgain \cite[Theorem~1']{Bou86}.
Our proof will closely follow Bourgain's proof of that theorem, replacing circles with arcs of the curves $v=a u^\beta$ and also invoking Bourgain's results on generalized circular maximal functions in the plane \cite{Bou86a}.

\smallskip
Theorems~\ref{thm:main1} and \ref{thm:main2} might also be interesting because they initiate the study of strong-type (a.k.a.\@ Bourgain-type) results for finite curved Euclidean configurations, asserting their existence in $A$ for a whole interval $I$ of parameters/scales. 
The two-point pattern \eqref{eq:pointparabmany} studied here could possibly be replaced with larger and more complicated configurations in the future.

%%%%%%%%%%

\section{Analytical reformulation}

It is sufficient to study the case $\beta>1$.
Afterwards, one can cover $0<\beta<1$ simply by interchanging the roles of the coordinate axes and applying the previous case to $1/\beta$.
Note that all bounds formulated in Theorem~\ref{thm:main1} and the statement of Theorem~\ref{thm:main2} are sufficiently symmetric to allow such swapping.
Thus, let us fix the parameter $\beta\in(1,\infty)$.

It is geometrically evident that one can realize an arc of the power curve $v=u^\beta$ as a part of a smooth closed simple curve $\Gamma$, which has non-vanishing curvature and which is the boundary of a centrally symmetric convex set in the plane.
More precisely, take parameters $0<\eta<\theta$ such that
\[ \Big(\frac{\theta}{\eta}\Big)^\beta - \beta\frac{\theta}{\eta} < \beta-1. \]
Figure~\ref{fig:curveset} depicts how the arc
\begin{equation}\label{eq:thearc}
\big\{ (u,u^\beta) : \eta\leq u\leq\theta \big\} 
\end{equation}
can be extended by its tangents at the endpoints to a boundary of a centrally symmetric convex set. It is then easy to curve and smooth this boundary a little in order to make it $\textup{C}^\infty$ with non-vanishing curvature while still containing the above arc.
The trick of realizing a power arc as a part of the boundary of an appropriate centrally symmetric convex set with intention of applying Bourgain's results \cite{Bou86a} has already been used by Marletta and Ricci \cite[Section~1, p.~59]{MR98}.

\begin{figure}[ht]
\includegraphics[width=0.45\linewidth]{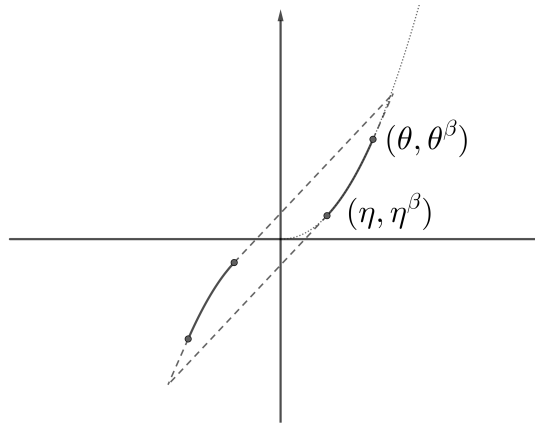}
\caption{The power arc, the reflected arc, and the tangents.}
\label{fig:curveset}
\end{figure}

Define $\nu$ to be the arclength measure of $\Gamma$. We can also parametrize the curve $\Gamma$ by arc length (i.e., traversing it at unit speed) as
\[ \Gamma = \{ (\gamma_1(s), \gamma_2(s)): s\in [0,L)\}, \] 
so that we have 
\[ \int_{\mathbb{R}^2} f(u,v) \,\textup{d}\nu(u,v) = \int_{0}^{L} f(\gamma_1(s), \gamma_2(s)) \,\textup{d}s \]
for every bounded measurable function $f$.
Now take a nonnegative smooth function $\Psi$ such that its support intersects $\Gamma$ precisely in the arc \eqref{eq:thearc}, and which is constant $1$ on a major part of that arc.  
Let $\sigma$ be the measure given by
\[ \textup{d}\sigma = \frac{\Psi\,\textup{d}\nu}{\int_\Gamma \Psi\,\textup{d}\nu}; \]
note that it is normalized as $\sigma(\mathbb{R}^2)=\sigma(\Gamma)=1$.
Then
\[ \int_{\mathbb{R}^2} f(u,v) \,\textup{d}\sigma(u,v) = \int_{\mathbb{R}} f(u,u^\beta) \psi(u) \,\textup{d}u \]
for every bounded measurable function $f$, where $\psi(u)$ is a constant multiple of 
\[ \Psi(u,u^\beta) (\gamma_1^{-1})'(u). \]
Thus, $\psi$ is a nonnegative $\textup{C}^\infty$ function whose support is contained in $[\eta,\theta]$.
All constants appearing in the proof are allowed to depend on $\Gamma,\beta,\eta,\theta,\Psi$ without further mention.

If $\sigma_t$ is the dilate of $\sigma$ by a number $t>0$, i.e., $\sigma_t(E):=\sigma(t^{-1}E)$, then we have
\[ \int_{\mathbb{R}^2} f(u,v) \,\textup{d}\sigma_{t}(u,v) = \frac{1}{t} \int_{\mathbb{R}} f\Big(u,\frac{u^{\beta}}{t^{\beta-1}}\Big) \psi\Big(\frac{u}{t}\Big)\,\textup{d}u, \]
so $\sigma_{t}$ is ``detects'' points on the curve $v=au^\beta$, where
\begin{equation}\label{eq:atsubst}
a = t^{1-\beta}.
\end{equation}
Finally, let $\tilde{\sigma}$ be the reflection of $\sigma$, i.e., $\tilde{\sigma}(E):=\sigma(-E)$.
Note that
\begin{equation}\label{eq:convo1}
\big(\tilde{\sigma}_t\ast f\big)(x,y) =  \frac{1}{t} \int_{\mathbb{R}} f\Big(x + u, y + \frac{u^{\beta}}{t^{\beta-1}}\Big) \psi\Big(\frac{u}{t}\Big)\,\textup{d}u.
\end{equation}

\smallskip
Both theorems will be consequences of the following purely analytical result.
Let $\mathbbm{1}_E$ denote the indicator function of a set $E\subseteq\mathbb{R}^2$.

\begin{proposition}\label{prop:main}
Take $0<\delta\leq1/2$ and a measurable set $A\subseteq[0,1]^2$ of measure $|A|\geq\delta$. Suppose that there exist dyadic numbers (i.e., elements of $2^\mathbb{Z}$)
\[ 1 > b_1 > c_1 > b_2 > c_2 > \cdots > b_J > c_J > 0 \]
having the property
\begin{equation}\label{eq:convo2}
\inf_{t\in[c_j,b_j]}\big(\tilde{\sigma}_t \ast \mathbbm{1}_A\big)(x,y)=0
\end{equation}
for every point $(x,y)\in A$ and every index $1\leq j\leq J$.
Then $J\leq \delta^{-C'}$ for some constant $C'\geq1$ independent of $\delta$ or $A$.
\end{proposition}

Our main task is to establish Proposition~\ref{prop:main} and its proof will span over the next section.

%%%%%%%%%%

\section{Proof of Proposition~\ref{prop:main}}

Let us write $A\lesssim B$ and $B\gtrsim A$ if the inequality $A\leq CB$ holds for a  constant $C\in(0,\infty)$. This constant $C$ is always understood to depend on $\Gamma,\beta,\eta,\theta,\Psi$ from previous sections. 
Let $\tau>0$ be a fixed positive number and $\varrho>0$ a fixed dyadic number; their values will be small and they will be chosen later. 

Take a measurable set $A\subseteq[0,1]^2$ with $|A|\geq \delta$. We write
\[ f:=\mathbbm{1}_A \quad\text{and}\quad g:= \mathbbm{1}_{[0,1]^2} - f. \] 
If we take an index $j$ such that
\begin{equation}\label{eq:howlargeisj}
j > J_0 := \Big\lceil\frac{1}{2} \log_2 \frac{\mathop{\textup{diam}}\Gamma}{\tau}\Big\rceil,
\end{equation}
then
\[ b_j \frac{\mathop{\textup{diam}}\Gamma}{2} 
\leq 2^{-2j} \mathop{\textup{diam}}\Gamma < \tau, \]
so for every $(x,y)\in A\cap[\tau,1-\tau]^2$ and $t\in[c_j,b_j]$ we have
\[ \big(\tilde{\sigma}_t\ast \mathbbm{1}_{[0,1]^2} \big)(x,y) = \sigma_t(\mathbb{R}^2) = \sigma(\mathbb{R}^2) = 1. \]
For such points $(x,y)$ the assumption \eqref{eq:convo2} then implies
\[ f(x,y) \sup_{t\in[c_j,b_j]} \big(\tilde{\sigma}_t\ast g \big)(x,y) = 1, \]
which in turn leads to a lower bound
\begin{align}
\int_{\mathbb{R}^2}f \cdot  \sup_{t\in [c_j,b_j]}  (\tilde{\sigma}_t\ast g) 
& \geq \int_{A\cap[\tau,1-\tau]^2} f \cdot \sup_{t\in [c_j,b_j]} \big(\tilde{\sigma}_t\ast g \big) \nonumber \\
& = |A\cap[\tau,1-\tau]^2| 
\geq |A|-4\tau 
= \int_{\mathbb{R}^2} f -4\tau, \label{taubelow}
\end{align}
provided $j$ is chosen large enough that \eqref{eq:howlargeisj} holds.

Let $\varphi_t$ be the  Poisson kernel on $\R^2$, i.e.,
\[{\varphi_t}(x,y) := \frac{t}{2\pi(t^2 + x^2+y^2)^{3/2}}\]
for every $t>0$, where the normalization is chosen such that $\int_{\R^2}\varphi_t = 1$. For a bounded measurable function $h$ we  will write 
\[ P_t h =  \varphi_t\ast h. \]
Also, for $k\in \mathbb{Z}$ let $\mathbb{E}_k$ denote the martingale averages with respect to the dyadic filtration, i.e.,  
\[\mathbb{E}_k h := \sum_{|Q|=2^{-2k}} \Big( |Q|^{-1}\int_Q h \, \Big ) \mathbbm{1}_Q, \]
where $h\in \textup{L}^1_{\textup{loc}}(\R^2)$ and 
the sum is taken over all dyadic squares $Q$ in $\R^2$ of area $2^{-2k}$ (and sidelength $2^{-k}$).

Take $t\in [c_j,b_j]$ and $k_j = -\log_2(\varrho c_j)$, which is an integer. 
We decompose 
\begin{align*}
\tilde{\sigma}_t  \ast g  
& =  (\tilde{\sigma}_t  \ast g -  \tilde{\sigma}_t \ast \mathbb{E}_{k_j} g    ) \\
& \quad   +\, (\tilde{\sigma}_t  \ast \mathbb{E}_{k_j} g -  \tilde{\sigma}_t \ast P_{\varrho c_j}  g    )
    +\,  
( \tilde{\sigma}_t \ast P_{\varrho c_j}  g    -  \tilde{\sigma}_t \ast P_{\varrho^{-1} b_j}   g ) \\
& \quad   + \, ( \tilde{\sigma}_t \ast P_{\varrho^{-1} b_j}   g - P_{\varrho^{-1} b_j}   g  ) 
      + P_{\varrho^{-1} b_j}   g .
\end{align*}
Taking the triangle inequality and the supremum over $t$  gives
\begin{align}
\int f \cdot \sup_{t \in [c_j,b_j]}\, (\tilde{\sigma}_t  \ast g )  
& \leq \int f \cdot \sup_{t \in [c_j,b_j]}| \tilde{\sigma}_t  \ast ( g -    \mathbb{E}_{k_j}   g )     | \label{dec1} \\
& \quad + \int f \cdot \sup_{t \in [c_j,b_j]} |\tilde{\sigma}_t  \ast \mathbb{E}_{k_j} g -  \tilde{\sigma}_t \ast P_{\varrho c_j}  g    |  \label{dec2} \\
& \quad + \int f \cdot \sup_{t \in [c_j,b_j]} |  \tilde{\sigma}_t \ast P_{\varrho c_j}  g    -  \tilde{\sigma}_t \ast P_{\varrho^{-1} b_j}   g  | \label{dec3} \\
& \quad + \int f\cdot \sup_{t \in [c_j,b_j]} | \tilde{\sigma}_t \ast P_{\varrho^{-1} b_j}   g - P_{\varrho^{-1} b_j}   g  | \label{dec4} \\
& \quad + \int f \cdot   P_{\varrho^{-1} b_j}   g . \nonumber
\end{align}

We will estimate each of the terms separately, using H\"older's inequality.
For the first term on the right-hand side of \eqref{dec1} 
  we will use the bound
  \begin{equation}
      \label{firstmainbd}
        \Big\| \sup_{t\in[c_j,1)} | \tilde{\sigma}_t\ast (g - \mathbb{E}_{k_j}  g)| \Big\|_{\textup{L}^p(\mathbb{R}^2)} \leq C_1 \varrho^\alpha \|g\|_{\textup{L}^p(\mathbb{R}^2)}
  \end{equation}
whenever $p>2$, where $\alpha$ is a positive constant depending only on $p$. 
(Any fixed finite value of $p$ greater than $2$ will do.)
This bound will follow from the central estimate (10) in Bourgain's paper \cite{Bou86a},
which can be written in our notation as
\begin{equation}\label{eq:Bourgainsmagic}
\Big\|\sup_{t \in [2^{-n}, 2^{-n+1})}| \tilde{\sigma}_t \ast h |\Big\|_{\textup{L}^p(\mathbb{R}^2)}
\lesssim 2^{-\alpha(i-n)} \|h\|_{\textup{L}^p(\mathbb{R}^2)}
\end{equation}
whenever $\mathbb{E}_i h=0$, while $n\leq i$ are positive integers and $p,\alpha$ are as before.
Bourgain \cite[(10)]{Bou86a} actually formulated \eqref{eq:Bourgainsmagic} for the full arclength measure $\textup{d}\nu$, but the very same proof establishes it also for the smooth truncation $\Psi\,\textup{d}\nu$. In fact, Bourgain has already performed several decompositions of $\nu$ \cite[Sections 3--6]{Bou86a}, and an additional smooth angular finite decomposition of $\Gamma$ can be added freely to the proof of his upper bound \cite[(10)]{Bou86a}, making the proof insusceptible to a smooth truncation by $\Psi$.

In order to prove \eqref{firstmainbd}, let $d_j=-\log_2(c_j)$.
We split   
$[c_j,1)$ into dyadic intervals $[2^{-n}, 2^{-n+1})$, 
estimate the maximum in $n$ by the $\ell^p$-sum, 
write
\[ g-\mathbb{E}_{k_j}g = \sum_{m=0}^{\infty} \Delta_{m+k_j} g, \] 
where $\Delta_i  = \mathbb{E}_{i+1} - \mathbb{E}_{i}$, and use the triangle inequality, after which it suffices to show  
\begin{equation*}
   % \label{high2}
 \bigg\|  \bigg ( \sum_{n=1}^{d_j}  \Big( \sum_{m=0}^\infty   \, \sup_{t \in [2^{-n}, 2^{-n+1})}| \tilde{\sigma}_t \ast  \Delta_{m+k_j} g |  \Big)^{p} \bigg) ^{1/p} \bigg\|_{\textup{L}^p(\mathbb{R}^2)} \lesssim \varrho^\alpha \|g\|_{\textup{L}^p(\mathbb{R}^2)}. 
\end{equation*}
The left-hand side can be rewritten as
\[ \bigg( \sum_{n=1}^{d_j} \Big\| \sum_{m=0}^\infty \sup_{t \in [2^{-n}, 2^{-n+1})} | \tilde{\sigma}_t \ast \Delta_{m+k_j} g | \Big\|_{\textup{L}^p(\mathbb{R}^2)}^p \bigg)^{1/p} \]
and then estimated by Minkowski's inequality with
\[ \leq \sum_{m=0}^\infty \Big(\sum_{n=1}^{d_j} \Big\|\sup_{t \in [2^{-n}, 2^{-n+1})}| \tilde{\sigma}_t \ast \Delta_{m+k_j} g |\Big\|_{\textup{L}^p(\mathbb{R}^2)}^p \Big)^{1/p}. \]
Finally, the inequality \eqref{eq:Bourgainsmagic} with $i=m+k_j$ bounds this by  
\begin{align*}
& \lesssim \sum_{m=0}^\infty\Big(\sum_{n=1}^{d_j} 2^{-p\alpha(m+k_j-n)} \|\Delta_{m+k_j}g\|_{\textup{L}^p(\mathbb{R}^2)}^p  \Big)^{1/p} \\
& \lesssim \sum_{m=0}^\infty\Big(\sum_{n=1}^{d_j} 2^{-p\alpha(m+k_j-n)} \|g\|_{\textup{L}^p(\mathbb{R}^2)}^p  \Big)^{1/p} \\
& \lesssim 2^{\alpha(d_j-k_j)}\|g\|_{\textup{L}^p(\mathbb{R}^2)}
= \varrho^\alpha\|g\|_{\textup{L}^p(\mathbb{R}^2)}, 
\end{align*}
as desired.

To control \eqref{dec2} and \eqref{dec3} 
we use   Bourgain's maximal estimate in the plane \cite[Theorem 1]{Bou86a}, 
\[ \Big\| \sup_{t\in (0,\infty)} |    \tilde{\sigma}_t \ast h  |  \Big\|_{\textup{L}^p(\mathbb{R}^2)} \lesssim \|h\|_{\textup{L}^p(\mathbb{R}^2)} \]
for $p>2$.
Here it gives
\begin{equation}
     \label{mfbd2}
    \Big\| \sup_{t\in [ c_j,b_j ]} |    \tilde{\sigma}_t \ast P_{\varrho c_j}  g    -  \tilde{\sigma}_t \ast P_{\varrho^{-1} b_j}   g  |  \Big\|_{\textup{L}^p(\mathbb{R}^2)} \leq C_2\| P_{\varrho^{-1} b_j}  g  - P_{\varrho c_j}  g \|_{\textup{L}^p(\mathbb{R}^2)}
\end{equation}
and 
\begin{equation}
     \label{mfbd3}
    \Big\| \sup_{t\in [ c_j,b_j ]} |   \tilde{\sigma}_t \ast \mathbb{E}_{k_j}   g  -  \tilde{\sigma}_t \ast P_{\varrho c_j}  g     |  \Big\|_{\textup{L}^p(\mathbb{R}^2)} \leq C_2\|  P_{\varrho c_j}  g - \mathbb{E}_{k_j}   g     \|_{\textup{L}^p(\mathbb{R}^2)} 
\end{equation}
for an absolute constant $C_2$. 
 
To estimate \eqref{dec4}, we claim  that  for each $(x,y)\in \R^2$,  $j$,   and  $t\leq b_j$, 
 \begin{equation}\label{smallt}
\big| \big(\tilde{\sigma}_t\ast P_{\varrho^{-1} b_j}  g \big)(x,y) - \big( P_{\varrho^{-1} b_j}  g \big)(x,y) \big| \leq C_3\varrho
\end{equation}
for some absolute constant $C_3$.
To see this, 
we first use that 
\[ \big| \big(\tilde{\sigma}_t\ast P_{\varrho^{-1} b_j}  g \big)(x,y) - \big( P_{\varrho^{-1} b_j}  g \big)(x,y) \big| 
\leq \big\| ( \tilde{\sigma}_t\ast \varphi_{\varrho^{-1} b_j} ) - \varphi_{\varrho^{-1} b_j} \big\|_{\textup{L}^1(\R^2)} \|g\|_{\textup{L}^\infty(\mathbb{R}^2)} \]
for each $(x,y)\in \R^2$.
Since  $\|g\|_{\textup{L}^\infty(\mathbb{R}^2)}\leq 1$, it only remains to bound, using \eqref{eq:convo1},
\begin{align*}
& \big\| ( \tilde{\sigma}_t\ast \varphi_{\varrho^{-1} b_j} ) - \varphi_{\varrho^{-1} b_j} \big\|_{\textup{L}^{1}(\R^2)} \\
& = \int_{\R^2} \bigg| \frac{1}{t}\int_{\R} \bigg( \varphi_{\varrho^{-1} b_j} \Big(x+u, y+\frac{u^\beta}{t^{\beta-1}}\Big) - \varphi_{\varrho^{-1} b_j}(x,y) \bigg) \psi\Big (\frac{u}{t} \Big) \,\textup{d}u \bigg| \,\textup{d}(x,y) \\
& \leq \int_{\R^2}  \frac{1}{t}\int_{\R} \Big | \varphi_1 \Big(x+\frac{ u \varrho}{ b_j}, y+\frac{u^\beta\varrho}{ b_j t^{\beta-1}}\Big) - \varphi_1(x,y)  \Big | \,\psi\Big (\frac{u}{t} \Big ) \,\textup{d}u \,\textup{d}(x,y), 
\end{align*}
where we also changed variables in $x,y$. 
By the mean value theorem, the last display is 
\[\leq \int_{\R^2}  \frac{1}{t}\int_{\R}  | \nabla \varphi_1 (z,w)  |\Big|\Big(\frac{u\varrho}{ b_j} , \frac{u^\beta \varrho}{ b_j t^{\beta-1}} \Big)\Big| \,\psi\Big( \frac{u}{t} \Big) \,\textup{d}u \,\textup{d}(x,y) \]
for
\[ (z,w)=a(x,y) + (1-a)\Big(x+\frac{u \varrho}{ b_j}, y+\frac{u^\beta \varrho}{ b_j t^{\beta-1}}\Big) \]
and some $0<a<1$. This is further bounded by
\[\lesssim  \int_{\R^2}  \frac{1}{t}\int_{\R}  \big(1+|(x,y)|^2\big)^{-3/2} \Big |\Big (\frac{u\varrho}{ b_j} , \frac{u^\beta \varrho}{ b_j t^{\beta-1}} \Big )\Big | \,\psi\Big (\frac{u}{t} \Big ) \,\textup{d}u \,\textup{d}(x,y) , \] 
where we also used  $|u|\lesssim t\leq b_j<1$, and dominated a non-centered $|\nabla \varphi_1|$ by a centered integrable function.   
Integrating in $u$ and $(x,y)$ we obtain a bound by
$C_3\varrho$.

Therefore, using \eqref{taubelow} to obtain a   lower bound, estimates \eqref{firstmainbd}, \eqref{mfbd2}, \eqref{mfbd3}, \eqref{smallt} for upper bounds, and H\"older's inequality, we obtain
\begin{align}
\int_{\R^2} f - 4\tau 
\leq  C_1\varrho^\alpha   + C_2\|   P_{\varrho c_j}  g  - \mathbb{E}_{k_j} g \|_{\textup{L}^p(\mathbb{R}^2)} +  C_2\| P_{\varrho^{-1} b_j}  g  - P_{\varrho c_j}  g \|_{\textup{L}^p(\mathbb{R}^2)} & \nonumber \\
+ \, C_3 \varrho + \int_{\R^2} f \cdot  P_{\varrho^{-1} b_j}   g & ,  \label{eq:almostthere}
\end{align}
provided $j$ is large enough.

Next, 
\[\int_{\R^2} f \cdot   P_{\varrho^{-1} b_j}   g  = \int_{\R^2} f \cdot  P_{\varrho^{-1} b_j}   \mathbbm{1}_{[0,1]^2}  - \int_{\R^2} f \cdot   P_{\varrho^{-1} b_j}   f  \]
and we have 
\begin{equation}
    \label{est1}
     \int_{\R^2} f \cdot   P_{\varrho^{-1} b_j}   \mathbbm{1}_{[0,1]^2}    \leq \int_{\R^2} f
\end{equation}
and
\begin{equation}
    \label{est2}
     \int_{\R^2} f \cdot   P_{\varrho^{-1} b_j}   f \geq c_0 \Big( \int_{\R^2} f \Big )^2
\end{equation}
for some absolute constant $c_0>0$.
The estimate \eqref{est1} follows by the trivial $\textup{L}^\infty$ bound for the convolution. 
To see \eqref{est2}, we note that by the Cauchy-Schwarz inequality, for any $k\in \mathbb{Z}$,  
\[ \int_{\R^2} f \cdot \mathbb{E}_k f  \geq \Big( \int_{\R^2}  f \Big)^2  \]
 Then it remains to bound the martingale averages from above by the Poisson averages.   
The reader can find the details in the proof of Lemma~2.1 in \cite{DGR19}.
Therefore, from \eqref{eq:almostthere} and $\int_{\R^2}f=|A|\geq\delta$ we get
\begin{equation}
\label{mainbound}
c_0 \delta^2 - 4\tau 
\leq C_1\varrho^\alpha   + C_3\varrho
+ C_2\|   P_{\varrho c_j}  g  - \mathbb{E}_{k_j} g \|_{\textup{L}^p(\mathbb{R}^2)} +   C_2\| P_{\varrho^{-1} b_j}  g  - P_{\varrho c_j}  g \|_{\textup{L}^p(\mathbb{R}^2)},
\end{equation}
which will turn out useful provided that $\tau$ is small enough.

Furthermore, we claim that for $p>2$ and for any $J>J_0$ we have 
\begin{equation}
    \label{sqfct}
    \sum_{j=J_0+1}^J \| P_{\varrho^{-1} b_j}  g  - P_{\varrho c_j}  g \|_{\textup{L}^p(\mathbb{R}^2)}^p  \leq C_4 \big(\log_2\varrho^{-1}\big)^p \|g\|^p_{\textup{L}^p(\mathbb{R}^2)} \leq C_4 \big(\log_2\varrho^{-1}\big)^p
\end{equation}
and
\begin{equation}
    \label{sqfct2}
    \sum_{j=J_0+1}^J \|   P_{\varrho c_j}  g - \mathbb{E}_{k_{j}}g \|_{\textup{L}^p(\mathbb{R}^2)}^p  \leq C_4 \|g\|^p_{\textup{L}^p(\mathbb{R}^2)} \leq C_4
\end{equation}
with the constant $C_4$ independent of $J_0,J$.  
These will be consequences of boundedness on $\textup{L}^p(\R^2)$, $1<p<\infty$, of the square functions
\[S_1 h := \Big( \sum_{i\in \mathbb{Z}} \big|P_{2^{-i+1}} h - P_{2^{-i}} h \big|^2 \Big)^{1/2} \]
and 
\[ S_2 h := \Big( \sum_{i\in \mathbb{Z}} \big|P_{2^{-i}}  h - \mathbb{E}_{i} h\big|^2 \Big)^{1/2}. \]
Bound for $S_1$ follows from the classical Calder{\'o}n-Zygmund theory \cite[Subsections 6.1.3]{Gra14}, 
while boundedness of $S_2$ was proven by Jones, Seeger, and Wright \cite[Sections 3--4]{JSW08}.
In fact, the emphasis of the paper \cite{JSW08} was on more general dilation structures and more general martingales, while the square function estimate from the last display is essentially due to Calder\'{o}n; see \cite[Subsection~6.4.4]{Gra14}.
Now, \eqref{sqfct} follows by recalling $p>2$ and writing
{\allowdisplaybreaks
\begin{align*}
\bigg( \sum_{j=J_0+1}^J \| P_{\varrho^{-1} b_j}  g  - P_{\varrho c_j}  g \|_{\textup{L}^p(\mathbb{R}^2)}^p \bigg)^{1/p} 
& \leq (2\log_2\varrho^{-1}+1) \bigg( \sum_{i\in\mathbb{Z}} \| P_{2^{-i+1}}  g  - P_{2^{-i}}  g \|_{\textup{L}^p(\mathbb{R}^2)}^p \bigg)^{1/p} \\
& \lesssim (\log_2\varrho^{-1}) \bigg\| \Big( \sum_{i\in\mathbb{Z}} | P_{2^{-i+1}}  g  - P_{2^{-i}} g |^p \Big)^{1/p} \bigg\|_{\textup{L}^p(\mathbb{R}^2)} \\
& \leq (\log_2\varrho^{-1}) \|S_1 g\|_{\textup{L}^p(\mathbb{R}^2)}
\lesssim (\log_2\varrho^{-1}) \|g\|_{\textup{L}^p(\mathbb{R}^2)} .
\end{align*}
}
Similarly we deduce \eqref{sqfct2}:
\[ \bigg( \sum_{j=J_0+1}^J \|   P_{\varrho c_j}  g - \mathbb{E}_{k_{j}}g \|_{\textup{L}^p(\mathbb{R}^2)}^p \bigg)^{1/p} \leq \|S_2 g\|_{\textup{L}^p(\mathbb{R}^2)} \lesssim \|g\|_{\textup{L}^p(\mathbb{R}^2)}. \]

To be completely determined, one can simply take $p=3$.
From \eqref{sqfct} and \eqref{sqfct2} we conclude that there exists  $j\in \{J_0+1,\ldots, J\}$ such that 
\[ \| P_{\varrho^{-1} b_j}  g  - P_{\varrho c_j}  g \|_{\textup{L}^3(\mathbb{R}^2)} , \|  P_{\varrho c_j}  g - \mathbb{E}_{k_j} g \|_{\textup{L}^3(\mathbb{R}^2)}   \leq (2C_4 (J-J_0)^{-1})^{1/3}\log_2\varrho^{-1}. \]
Together with \eqref{mainbound} applied for this particular $j$ and $\tau=c_0\delta^2/8$ we obtain
\[ 2^{-1}c_0\delta^2 \leq C_1\varrho^\alpha + C_3\varrho + 2 C_2(2C_4(J-J_0)^{-1})^{1/3} \log_2\varrho^{-1}, \]
i.e.,
\[ J - J_0 \lesssim \Big(\frac{\log_2\varrho^{-1}}{2^{-1}c_0\delta^2-C_1\varrho^\alpha-C_3\varrho}\Big)^3. \]
Now we recall that we actually chose $J_0$ in \eqref{eq:howlargeisj} at the beginning of the proof, which guarantees that $J_0\leq  \log_2(C_5 \delta^{-2})$ for a suitable constant $C_5$.
Taking $\varrho$ to be a small multiple of $\min\{\delta^{2/\alpha},\delta^2\}$
we obtain $J \leq \delta^{-C'}$ for a suitable constant $C'$. 

%%%%%%%%%%

\section{Proofs of Theorems~\ref{thm:main1} and \ref{thm:main2}}

In this section we deduce the two main theorems from Proposition~\ref{prop:main}.
Once again, it is sufficient to consider $\beta\in(1,\infty)$.

\begin{proof}[Proof of Theorem~\ref{thm:main1}]
Set $J=\lfloor \delta^{-C'}\rfloor+1$, where $C'$ is the constant from Proposition~\ref{prop:main}.
Let us simply choose consecutive dyadic scales, $b_j=2^{-2j+1}$ and $c_j=2^{-2j}$ for every $1\leq j\leq J$.
By the contraposition of Proposition~\ref{prop:main} and using formula \eqref{eq:convo1} we conclude that there exist a point $(x,y)\in A$ and an index $1\leq j\leq J$ such that for every $c_j\leq t\leq b_j$ the set $A$ contains a point of the form
\begin{equation}\label{eq:pointparab2}
\Big(x + u, y + \frac{u^{\beta}}{t^{\beta-1}}\Big), \quad \eta t< u< \theta t.
\end{equation}
Substituting \eqref{eq:atsubst} we get
\begin{equation*}
c_j\leq t\leq b_j \quad\Longleftrightarrow\quad b_j^{1-\beta}\leq a\leq c_j^{1-\beta},
\end{equation*}
which now means that for every
\[ a \in I := \big[ 2^{(\beta-1)(2j-1)}, 2^{(\beta-1)2j} \big] \]
there exists
\[ \eta a^{-1/(\beta-1)} < u < \theta a^{-1/(\beta-1)} \]
such that $(x+u,y+au^\beta)\in A$. 
Observing 
\begin{align*}
& \inf I \geq 1, \\
& \sup I \leq 2^{2(\beta-1)J} \leq 2^{4(\beta-1)\delta^{-C'}}, \\
& |I| \geq 2^{\beta-1} - 1,
\end{align*}
and that any such $u$ satisfies
\begin{align*}
& u \geq \eta 2^{-2j} \geq \eta 2^{-2J} \geq \eta 2^{-2\delta^{-C'}}, \\
& u \leq \theta 2^{-2j+1} \leq \theta
\end{align*}
we finally establish Theorem~\ref{thm:main1}.
\end{proof}

\begin{proof}[Proof of Theorem~\ref{thm:main2}]
Suppose that the claim does not hold for some measurable set $A\subseteq\mathbb{R}^2$ with $\overline{\delta}(A)>0$. 
Take $\delta:=\overline{\delta}(A)/2$ and $J=\lfloor \delta^{-C'}\rfloor+1$, where $C'$ is the constant from Proposition~\ref{prop:main}.
Inductively we construct positive numbers
\begin{equation*}%\label{eq:bsandcs}
C_1 > B_1 > C_2 > B_2 > \cdots > C_J > B_J
\end{equation*}
satisfying $C_{j+1}\leq B_j/8^{\beta-1}$ and such that for each $j\geq1$ and every point $(x,y)\in A$ there exists $a\in[B_j,C_j]$ with the property that $A$ does not contain a point of the form 
\[ (x+u, y+au^\beta), \quad u>0. \]
After the change of variables \eqref{eq:atsubst} we see that for each $j\geq1$ and every point $(x,y)\in A$ there exists
\[ C_j^{-1/(\beta-1)} \leq t \leq B_j^{-1/(\beta-1)} \]
such that $A$ does not contain a point of the form \eqref{eq:pointparab2}, so 
\[ \big(\tilde{\sigma}_t \ast \mathbbm{1}_A\big)(x,y)=0. \]
By the definition of the upper Banach density there exist a number $R\geq B_J^{-1/(\beta-1)}$ and a point $(x_0,y_0)\in\mathbb{R}^2$ such that
\[ \big| A\cap \big([-R,R]^2 + (x_0,y_0)\big) \big| \geq \delta\cdot 4R^2. \]
Define
\begin{equation}\label{eq:primeset}
A' := \bigg( \frac{1}{2R} \big( A - (x_0,y_0) \big) + \Big(\frac{1}{2},\frac{1}{2}\Big) \bigg) \cap [0,1]^2
\end{equation}
and let $b_j$ and $c_j$ respectively be the number $B_{J+1-j}^{-1/(\beta-1)}/2R$ rounded up to the nearest dyadic number and the number $C_{J+1-j}^{-1/(\beta-1)}/2R$ rounded down to the nearest dyadic number, i.e.,
\begin{equation}\label{eq:primebscs}
b_j := 2^{\lceil \log_2(B_{J+1-j}^{-1/(\beta-1)}/2R)\rceil}, \quad c_j := 2^{\lfloor\log_2( C_{J+1-j}^{-1/(\beta-1)}/2R)\rfloor}
\end{equation}
for every $1\leq j\leq J$.
Finally, for every $(x,y)\in A'$ and every $1\leq j\leq J$ this implies
\[ \big(\tilde{\sigma}_t \ast \mathbbm{1}_{A'}\big)(x,y)=0 \]
for some $c_j\leq t\leq b_j$, while we have chosen $J$ so that $J>\delta^{-C'}$. 
Note that also $|A'|\geq\delta$, so the set \eqref{eq:primeset} and the numbers \eqref{eq:primebscs} violate Proposition~\ref{prop:main}, which leads us to a contradiction.
\end{proof}

%%%%%
 
\section*{Acknowledgments}
P. D. is partially supported by the NSF grant DMS-2154356. P. D. and V. K. were supported by the NSF grant DMS-1929284 while the authors were in residence at the Institute for Computational and Experimental Research in Mathematics in Providence, RI, during the Harmonic Analysis and Convexity program. 
V. K. and M. S. are partially supported by the Croatian Science Foundation project UIP-2017-05-4129 (MUNHANAP).
M. S. is supported  by a fellowship through the Grand Challenges Initiative at Chapman University.

%%%%%%%%%%%%%%%%%%%%%%%%%%%%%%%%%%%%%%%%%%%%%%%%

\bibliography{parabolae}{}
\bibliographystyle{plain}

\end{document}